\documentclass[11pt]{amsart}

\usepackage{amsmath,amssymb,amsfonts,amsthm}
\usepackage{microtype}
\usepackage{mathrsfs}
\usepackage[colorlinks=true,linkcolor=blue,citecolor=blue,urlcolor=blue]{hyperref}
\hypersetup{
  pdftitle={An Orlicz variational formula for David-type Beltrami equations},
  pdfauthor={Ryo Matsuda}, bookmarksnumbered=true, linktoc=all
}
\usepackage[nameinlink,capitalize]{cleveref}
\usepackage{aliascnt}

\theoremstyle{plain}
\newtheorem{theorem}{Theorem}[section]
\newtheorem{introthm}{Theorem}

\crefname{introthm}{Theorem}{Theorems}
\Crefname{introthm}{Theorem}{Theorems}

\newaliascnt{mainthm}{theorem}
\newtheorem{mainthm}[mainthm]{Main Theorem}
\aliascntresetthe{mainthm}

\newaliascnt{proposition}{theorem}
\newtheorem{proposition}[proposition]{Proposition}
\aliascntresetthe{proposition}

\newaliascnt{lemma}{theorem}
\newtheorem{lemma}[lemma]{Lemma}
\aliascntresetthe{lemma}

\crefname{mainthm}{Main Theorem}{Main Theorems}
\Crefname{mainthm}{Main Theorem}{Main Theorems}
\crefname{theorem}{Theorem}{Theorems}
\Crefname{theorem}{Theorem}{Theorems}
\crefname{proposition}{Proposition}{Propositions}
\Crefname{proposition}{Proposition}{Propositions}
\crefname{lemma}{Lemma}{Lemmas}
\Crefname{lemma}{Lemma}{Lemmas}

\newcommand{\C}{\mathbb C}
\newcommand{\R}{\mathbb R}
\newcommand{\D}{\mathbb D}
\newcommand{\id}{\mathrm{id}}
\newcommand{\loc}{\mathrm{loc}}
\newcommand{\dA}{\,dA}
\newcommand{\dbar}{\bar\partial}
\newcommand{\dz}{\partial}
\newcommand{\Sol}{\mathfrak S}
\newcommand{\Cauchy}{\mathcal C}
\newcommand{\Beurling}{\mathcal B}
\newcommand{\op}{\mathrm{op}}
\newcommand{\Lip}{\mathrm{Lip}}
\newcommand{\BMO}{\mathrm{BMO}}
\newcommand{\HL}{\mathcal M} % Hardy--Littlewood maximal operator
\newcommand{\Berslog}{\mathscr B_{\log,0}}

\title[An Orlicz variational formula for David-type Beltrami equations]
{An Orlicz variational formula for David-type Beltrami equations}

\author{Ryo Matsuda}
\address{Department of Mathematical Sciences, College of Science and Engineering\\
Ritsumeikan University, 1-1-1 Nojihigashi, Kusatsu, Shiga 525-8577, Japan}
\email{r-mat@fc.ritsumei.ac.jp}


\begin{document}

\begin{abstract}
We develop a Banach-space variational theory for degenerate Beltrami equations of David type. For coefficients supported in a bounded domain $\mathcal U\subset\C$, the pointwise change of variables $\nu=2\mu/(1-|\mu|^2)$ identifies the David coefficients with the exponential Orlicz space $L^\Phi_{\mathcal U}(\C)$ for $\Phi(s)=e^s-s-1$. On the open subset of parameters with a sufficiently large exponential moment, the principal solution map is locally real $C^{1,1}$ into $W^{1,2}_{\mathrm{loc}}(\C)$, with Fr\'echet derivative solving the linearized Beltrami equation. On the Orlicz heart $E^\Phi_{\mathcal U}(\C)$ it is real $C^\infty$, yet fails to be real analytic at zero: along any direction $h$, its Taylor series at zero has radius of convergence $1/\|h\|_\infty$ in $W^{1,2}(B)$ for every disk $B\supset\overline{\mathcal U}$. Finally, for $\mathcal U=\D$, the exterior Schwarzian map from $E^\Phi_{\D}(\C)$ to the little logarithmic Bers space is continuous, but fails to be Fr\'echet differentiable at zero; moreover, its zero fiber is not a $C^1$ embedded Banach submanifold near zero, obstructing any local $C^1$ split submersion structure on the resulting quotient.
\end{abstract}

\maketitle

\section{Introduction}\label{sec:introduction}

Given a measurable function \(\mu\) on \(\C\) with \(|\mu|<1\) almost
everywhere, the planar Beltrami equation
\[
   \dbar f=\mu\,\dz f
\]
relates elliptic partial differential equations, quasiconformal analysis,
and the moduli theory of Riemann surfaces. In the uniformly elliptic
regime \(\|\mu\|_\infty<1\), the measurable Riemann mapping theorem gives
a unique homeomorphic solution after normalization \cite{AB,Mor}.
Ahlfors and Bers proved that the normalized solution depends
holomorphically on \(\mu\) in the open unit ball of \(L^\infty\)
\cite{AB}. This dependence underlies the Bers embedding, which realizes
the Teichm\"uller space of a hyperbolic Riemann surface as a domain in a
complex Banach space of holomorphic quadratic differentials.

For surfaces of infinite type, degenerate Beltrami coefficients occur in
constructions of points on the Bers boundary. One class admitting
homeomorphic solutions is given by the \emph{David condition}
\cite{David}. For a coefficient of bounded support, its distortion
\(K_\mu=(1+|\mu|)/(1-|\mu|)\) satisfies this condition precisely when
there are \(C,c>0\) such that
\[
    |\{z:K_\mu(z)>t\}|\le Ce^{-ct}\qquad(t\ge1).
\]
Homeomorphic solutions are called \emph{David maps}. They occur in
parabolic surgery in complex dynamics \cite{Hai}. In Teichm\"uller
theory, David deformations give rise to infinite-dimensional families
on Bers boundaries of infinite-type surfaces \cite{Mat1} and to examples
of boundary points that cannot be approximated by maximal cusps
\cite{Mat2}.

These examples motivate two questions. How does a principal solution
vary with a David coefficient? What smooth structure can this variation
induce on the quotient defined by its exterior Schwarzian? At a
coefficient with \(\|\mu\|_\infty=1\), an \(L^\infty\) neighborhood
contains perturbations for which \(|\mu|<1\) fails on a set of positive
area. In the Orlicz coordinates used below, the distortions arising from
parameters in any Luxemburg ball of positive radius about zero have no
common almost-everywhere finite majorant. To control the variation
uniformly on a small ball, we instead use integral bounds for these
distortions.

Fix a bounded domain \(\mathcal U\Subset\C\). We use
\(L^\Phi_{\mathcal U}(\C)\) for the exponential Orlicz space of
functions vanishing almost everywhere off \(\mathcal U\), with
\(\Phi(s)=e^s-s-1\) and Luxemburg norm \(\|\cdot\|_\Phi\). The
pointwise real-analytic change of variables
\[
   \nu=\frac{2\mu}{1-|\mu|^2},\qquad
   \mu=F(\nu)=\frac{\nu}{1+\sqrt{1+|\nu|^2}}
\]
identifies the David coefficients vanishing off \(\mathcal U\) with
\(L^\Phi_{\mathcal U}(\C)\) (\cref{prop:david-orlicz}). Write
\[
   p(\nu)=\sup\left\{a>0:\int_{\mathcal U}e^{a|\nu|}\dA<\infty\right\}.
\]
The corresponding exponential integrability index for \(K_{F(\nu)}\)
is \(p(\nu)/2\), by \eqref{eq:distortion-exponent}. Denote the principal
solution by \(\Sol(\nu)=f_\nu=f^{F(\nu)}\), with normalization
\(f_\nu(z)=z+O(1/z)\) as \(z\to\infty\).

Our first result concerns the open high-moment domain
\(\mathscr D_\Lambda=\{\nu\in L^\Phi_{\mathcal U}(\C):p(\nu)>\Lambda\}\).
We fix any \(\Lambda>2\max\{p_{\mathrm{res}},8\}\), where
\(p_{\mathrm{res}}\) is a sufficient exponent for the weighted \(L^2\)
resolvent estimate of \cref{thm:prelim-aim}.

\begin{introthm}[\cref{thm:intro-main}]\label{intro:variation}
The principal solution map
\(\Sol:\mathscr D_\Lambda\to W^{1,2}_{\loc}(\C)\) is locally real
\(C^{1,1}\) in each local Sobolev norm. For \(\nu\in\mathscr D_\Lambda\)
and \(\eta\in L^\Phi_{\mathcal U}(\C)\), its Fr\'echet derivative
\(V=D\Sol_\nu[\eta]\) is the unique \(W^{1,2}_{\loc}(\C)\) solution of
\[
   (\dbar-F(\nu)\dz)V=DF_\nu(\eta)\,\dz f_\nu,\qquad
   V(z)=O(1/z)\quad(z\to\infty),
\]
where \(DF_\nu(\eta)\) is the pointwise real differential.
\end{introthm}

The proof of \cref{intro:variation} combines the coefficient-dependent
BMO majorant and weighted resolvent estimate of Astala--Iwaniec--Martin
\cite{AIM} with the Sobolev regularity and area distortion estimates of
Astala--Gill--Rohde--Saksman \cite{AGRS}. On a sufficiently small
Luxemburg ball, the majorants have a uniform exponential moment and the
Jacobians have a uniform logarithmic moment. Solving the linearized
equation on the target plane and pulling its solution back by \(f_\nu\)
controls the first derivative. Applying the weighted resolvent to the
remainder gives an error of order \(\|h\|_\Phi^2\).

To control variations of successively higher order, our estimates use
exponential moments of increasing order. This leads to the \emph{Orlicz
heart}
\[
  E^\Phi_{\mathcal U}(\C)
    =\{\nu\in L^\Phi_{\mathcal U}(\C):p(\nu)=\infty\}
    =\overline{L^\infty_{\mathcal U}(\C)}^{\,\|\cdot\|_\Phi},
\]
a closed subspace of \(L^\Phi_{\mathcal U}(\C)\) in which bounded
parameters are dense \cite{RaoRen}.

\begin{introthm}[\cref{thm:intro-smooth,prop:directional-radius}]\label{intro:heart}
On the heart \(E^\Phi_{\mathcal U}(\C)\), the map \(\Sol\) is real
Fr\'echet \(C^\infty\) in each local \(W^{1,2}\) norm. At the origin,
for \(\eta,h,k\in E^\Phi_{\mathcal U}(\C)\),
\[
   D\Sol_0[\eta]=\tfrac12\Cauchy\eta,\qquad
   D^2\Sol_0[h,k]=\tfrac14\Cauchy\bigl(h\,\Beurling k+k\,\Beurling h\bigr),
\]
where \(\Cauchy\) and \(\Beurling\) are the Cauchy and Beurling transforms.
For every bounded disk \(B\supset\overline{\mathcal U}\), the Taylor
series of \(t\mapsto\Sol(th)|_B\) at zero in \(W^{1,2}(B)\) has radius
\(1/\|h\|_\infty\), with \(1/0=\infty\) and \(1/\infty=0\).
In particular, \(\Sol\) is not real analytic at zero on the heart.
\end{introthm}

In view of \cref{intro:heart}, we ask whether the smooth dependence of
the solution descends to the corresponding quotient. Take
\(\mathcal U=\D\), write \(E=E^\Phi(\D)\), and let
\(\D^*=\{z\in\C:|z|>1\}\cup\{\infty\}\). The exterior Schwarzian
\(\beta(\nu)=S(f_\nu|_{\D^*})\) depends smoothly on \(\nu\) in the
compact-open topology. Together with the Kraus--Nehari bound
\cite{Nehari}, this gives continuity into the little logarithmic Bers
space \(\Berslog(\D^*)\) in its norm \(\|\cdot\|_{\log}\). The
compact-open differential \(\mathcal A\) at zero defines a bounded
complex-linear surjection from \(E\) onto \(\Berslog(\D^*)\) with a
bounded complex-linear right inverse (\cref{prop:infinitesimal-bers}).

\begin{introthm}[\cref{thm:trivial-fiber-obstruction}]\label{intro:quotient}
The map \(\beta:E\to\Berslog(\D^*)\) fails to be Fr\'echet
differentiable at zero in the logarithmic norm. Its trivial fiber
\(Z=\beta^{-1}(0)\) is not a \(C^1\) embedded Banach submanifold near
zero. For the equivalence relation \(\nu\sim\eta\) defined by
\(\beta(\nu)=\beta(\eta)\), no Banach manifold structure on
\(E/{\sim}\) makes the quotient map a real \(C^1\) split submersion
near the origin.
\end{introthm}

To prove \cref{intro:quotient}, we conjugate a compactly supported
deformation that equals the identity on \(\D^*\) by disk automorphisms
moving its support toward \(\partial\D\). Its conjugates give bounded
parameters \(\nu_j\in Z\) for which \(\|\nu_j\|_\Phi\to0\) and
\(\|\mathcal A\nu_j\|_{\log}\ge c\|\nu_j\|_\Phi>0\)
(\cref{lem:boundary-sequence}). Since \(\beta(\nu_j)=\beta(0)\), the
linear approximation with derivative \(\mathcal A\) has an error of at
least first order along this sequence.

\Cref{sec:preliminaries} recalls the David existence theorem and the
analytic estimates used below. \Cref{sec:orlicz-coordinates} develops
the Orlicz coordinates and proves the uniform weighted resolvent and
Jacobian estimates. \Cref{sec:frechet} proves \cref{intro:variation},
and \cref{sec:smooth-heart} proves \cref{intro:heart} by extending
higher variations from bounded parameters by density and computing the
directional Taylor radii from the Beltrami equation.
\Cref{sec:bers-applications} constructs the right inverse of
\(\mathcal A\) and proves \cref{intro:quotient}.

\section{Preliminaries}\label{sec:preliminaries}

\subsection{David coefficients and principal solutions}
\label{subsec:david-solutions}

Fix a bounded domain \(\mathcal U\Subset\C\). All coefficients vanish
almost everywhere off \(\mathcal U\), unless a different support is
specified. Write \(dA=dx\,dy\) and \(|E|=\int_E dA\).
For a measurable coefficient \(\mu\) with \(|\mu|<1\) almost everywhere,
put \(K_\mu=(1+|\mu|)/(1-|\mu|)\). Following David \cite{David}, we say
that \(\mu\) is of \emph{David type} if there are \(C,c>0\) and
\(0<\varepsilon_0<1\) such that
\begin{equation}\label{eq:david-tail}
  |\{|\mu|>1-\varepsilon\}|\le C e^{-c/\varepsilon}
  \qquad(0<\varepsilon<\varepsilon_0).
\end{equation}
Since \(\{|\mu|>1-\varepsilon\}=\{K_\mu>2/\varepsilon-1\}\),
this is equivalent, after changing the constants, to
\(|\{K_\mu>t\}|\le C'e^{-c't}\) for \(t\ge1\). It is also equivalent to
\begin{equation}\label{eq:david-moment}
  \int_{\mathcal U}e^{qK_\mu}\dA<\infty
  \qquad\text{for some }q>0.
\end{equation}
Indeed, integrating the distortion tail gives \eqref{eq:david-moment}
for every \(0<q<c'\), and Chebyshev's inequality gives the converse.
Here the moment is taken over \(\mathcal U\), since \(K_\mu=1\) off
\(\mathcal U\).

\begin{theorem}[David's mapping theorem]
\label{thm:prelim-david}
If \(\mu\) satisfies \eqref{eq:david-moment}, the equation
\(\dbar f=\mu\dz f\) admits a unique principal solution \(f^\mu\) in
\(W^{1,Q}_{\loc}(\C)\), where \(Q(t)=t^2/\log(e+t)\). It is a
homeomorphism of \(\C\) and is normalized by
\(f^\mu(z)=z+O(1/z)\) as \(z\to\infty\).
\end{theorem}

This is \cite[Theorem~1.3]{AGRS} after a dilation of a disk containing
\(\overline{\mathcal U}\). David's original normalization fixes
\(0,1,\infty\) \cite{David}.

\subsection{Sobolev regularity and area distortion}
\label{subsec:agrs-estimates}

\begin{theorem}[Astala--Gill--Rohde--Saksman]\label{thm:prelim-agrs}
Let \(f=f^\mu\) be as in \cref{thm:prelim-david}, and let \(q>0\)
be an exponent for which \eqref{eq:david-moment} holds. Both \(f\)
and \(f^{-1}\) take sets of area zero to sets of area zero. If \(q>1\),
then \(f\in W^{1,2}_{\loc}(\C)\).

In the unit disk normalization \(\mu=0\) off \(\D\), put
\(M_q=\int_\D e^{qK_\mu}\dA\). For each \(0<s<q\) and every measurable
\(E\subset\D\) of positive area,
\begin{equation}\label{eq:agrs-area}
   |f(E)|\le C(q,s,M_q)\,[\log(e+1/|E|)]^{-s}.
\end{equation}
The constant can be chosen uniformly under a fixed upper bound for \(M_q\).
\end{theorem}

The null-set properties follow from \cite[Corollary~4.4 and the remark
following it]{AGRS}. The Sobolev assertion is \cite[Corollary~1.2]{AGRS}
and \eqref{eq:agrs-area} is \cite[Theorem~5.1]{AGRS}.
In particular, the area formula gives
\(\int_E J_f\dA=|f(E)|\) for measurable \(E\). We will use this identity
to estimate the source of the linearized equation on the target plane.

\subsection{The weighted Beltrami estimate}\label{subsec:aim-input}

For compactly supported \(\omega\in L^2(\C)\) we use the conventions
\[
   \Cauchy\omega(z)=\frac1\pi\int_\C\frac{\omega(w)}{z-w}\dA(w),
   \qquad \Beurling\omega=\dz\Cauchy\omega.
\]
Thus \(\dbar\Cauchy\omega=\omega\) in distributions and
\(\Beurling\) is an isometry on \(L^2(\C)\).
The following form of the AIM estimate specifies the weight acting on
the right-hand side of the inhomogeneous equation.

\begin{theorem}[Astala--Iwaniec--Martin]\label{thm:prelim-aim}
There is a universal \(p_{\mathrm{res}}\ge1\) with the following
property. Suppose \(\mu=0\) off \(\D\) and
\(\int_\D e^{qK_\mu}\dA<\infty\) for some \(q>p_{\mathrm{res}}\).
Fix \(p_{\mathrm{res}}<\alpha<q\), set \(k=K_\mu\chi_{\D}\), and let
\(\HL\) be the Hardy--Littlewood maximal operator on \(\C\). Write \([\cdot]_{\BMO(\C)}\) for the seminorm of bounded mean
oscillation. Put
\begin{equation}\label{eq:aim-majorant}
  G_{\alpha,k}(z)=\frac{e^{\alpha k(z)}}{(1+|z|)^3},\qquad
  k^*_\alpha(z)=\frac1\alpha
    \log\!\left((1+|z|)^3\HL G_{\alpha,k}(z)\right).
\end{equation}
Then \(k^*_\alpha\ge k\) almost everywhere and
\([k^*_\alpha]_{\BMO(\C)}\le C/\alpha\) for a universal \(C\).
If \(H\) vanishes off \(\D\) and \(k^*_\alpha H\in L^2(\C)\), the
equation \((I-\mu\Beurling)\omega=H\) has a unique solution in
\(L^2(\C)\), and
\begin{equation}\label{eq:AIM-resolvent-4}
   \|\omega\|_2\le4\|k^*_\alpha H\|_2 .
\end{equation}
\end{theorem}

The majorant and its BMO seminorm bound are given in
\cite[Corollary~20.4.2 and (20.36)]{AIM}. The resolvent estimate is
\cite[Theorem~20.4.3]{AIM}, whose proof uses the BMO estimate for
Jacobians in \cite[Theorem~20.4.4]{AIM}. The weight in
\eqref{eq:AIM-resolvent-4} depends on \(\mu\).
We establish its local exponential moment bounds for families in
\cref{subsec:quantitative-AIM}.

\section{Orlicz coordinates and principal estimates}
\label{sec:orlicz-coordinates}

\subsection{The parameter space}

For \(\Phi(s)=e^s-s-1\), let
\begin{equation}\label{eq:Luxemburg}
  \|u\|_{L^\Phi(E)}
  =\inf\left\{\lambda>0:
    \int_E\Phi(|u|/\lambda)\dA\le1\right\}.
\end{equation}
We identify \(L^\Phi(\mathcal U)\) with the space
\(L^\Phi_{\mathcal U}(\C)\) of its zero extensions and abbreviate
its Luxemburg norm to \(\|\cdot\|_\Phi\). For
\(\nu\in L^\Phi_{\mathcal U}(\C)\) define
\[
  p(\nu)=\sup\left\{a>0:
       \int_{\mathcal U}e^{a|\nu|}\dA<\infty\right\}.
\]
The Orlicz heart is the closed subspace \cite{RaoRen}
\begin{equation}\label{eq:intro-heart}
 E^\Phi_{\mathcal U}(\C)
 =\{\nu:p(\nu)=\infty\}
 =\overline{L^\infty_{\mathcal U}(\C)}^{\,\|\cdot\|_\Phi}.
\end{equation}
The subscript \(\mathcal U\) denotes vanishing
almost everywhere off \(\mathcal U\).

\begin{lemma}[Exponential moments]\label{lem:luxemburg-exponential}
On a set \(E\) of finite area, \(L^\Phi(E)\) consists precisely of the
functions \(u\) with \(e^{a|u|}\in L^1(E)\) for some \(a>0\).
For \(0<\theta\le1\),
\begin{equation}\label{eq:unit-ball-exp}
  \sup_{\|u\|_{L^\Phi(E)}\le1}\int_Ee^{\theta|u|}\dA
  \le C(|E|+1).
\end{equation}
\end{lemma}

\begin{proof}
Monotone convergence in \eqref{eq:Luxemburg} gives
\(\int_E\Phi(|u|)\dA\le1\) on the closed unit ball. Since
\(e^s/(1+\Phi(s))\) is bounded on \([0,\infty)\),
\eqref{eq:unit-ball-exp} follows. Scaling gives a positive
exponential moment for every \(u\in L^\Phi(E)\). Conversely,
\(\Phi(|u|/\lambda)\le e^{a|u|}\) for \(\lambda\ge1/a\), so an
integrable exponential moment and dominated convergence imply
\(\int_E\Phi(|u|/\lambda)\dA\to0\).
\end{proof}

Our pointwise coordinates are
\begin{equation}\label{eq:intro-coordinate}
  \nu=\frac{2\mu}{1-|\mu|^2},\qquad
  \mu=F(\nu)=\frac{\nu}{1+\sqrt{1+|\nu|^2}}.
\end{equation}
We write \(\mu_\nu=F(\nu)\).

\begin{proposition}[David coefficients]\label{prop:david-orlicz}
The coordinate \eqref{eq:intro-coordinate} identifies
\(L^\Phi_{\mathcal U}(\C)\) with the David coefficients vanishing
almost everywhere off \(\mathcal U\). More precisely,
\begin{equation}\label{eq:distortion-exponent}
 \sup\left\{q>0:\int_{\mathcal U} e^{qK_{F(\nu)}}\dA<\infty\right\}
 =\tfrac12 p(\nu).
\end{equation}
\end{proposition}

\begin{proof}
For \(r=|\nu|\), substitution gives
\begin{equation}\label{eq:distortion-comparison}
  K_{F(\nu)}=\sqrt{1+r^2}+r,\qquad
  K_{F(\nu)}-2r=\frac{1}{\sqrt{1+r^2}+r}\in(0,1].
\end{equation}
Thus \(e^{2qr}\le e^{qK_{F(\nu)}}\le e^q e^{2qr}\) for every \(q>0\).
This proves \eqref{eq:distortion-exponent}. The David condition is
therefore equivalent to the existence of a positive exponential moment
of \(r\) by \eqref{eq:david-moment}, or equivalently to
\(\nu\in L^\Phi_{\mathcal U}(\C)\) by \cref{lem:luxemburg-exponential}.
The inverse coordinate in \eqref{eq:intro-coordinate} gives the bijection.
\end{proof}

With \(p_{\mathrm{res}}\) as in \cref{thm:prelim-aim}, we fix
\begin{equation}\label{eq:intro-Lambda}
 \Lambda>2\max\{p_{\mathrm{res}},8\},\qquad
 \mathscr D_\Lambda=\{\nu\in L^\Phi_{\mathcal U}(\C):p(\nu)>\Lambda\}.
\end{equation}

By \eqref{eq:distortion-exponent}, each \(\nu\in\mathscr D_\Lambda\)
admits a distortion exponent \(q>\max\{p_{\mathrm{res}},8\}\).

\begin{lemma}[Local uniformity of moments]\label{lem:small-ball-moment}
If \(0<q<p(\nu_*)\), there exist \(r,C>0\) such that
\begin{equation}\label{eq:small-ball-moment}
 \sup_{\|\nu-\nu_*\|_\Phi<r}
 \int_{\mathcal U}e^{q|\nu|}\dA\le C .
\end{equation}
In particular, \(\mathscr D_\Lambda\) is open and contains the heart.
\end{lemma}

\begin{proof}
Choose \(a>1\) with \(qa<p(\nu_*)\), put \(a'=a/(a-1)\), and take
\(r\le1/(qa')\). For \(\|h\|_\Phi<r\), H\"older's inequality gives
\[
 \int_{\mathcal U}e^{q|\nu_*+h|}\dA
 \le\left(\int_{\mathcal U}e^{qa|\nu_*|}\dA\right)^{1/a}
    \left(\int_{\mathcal U}e^{qa'|h|}\dA\right)^{1/a'}.
\]
The second factor is uniformly bounded by \eqref{eq:unit-ball-exp}
applied to \(h/r\). Taking \(\Lambda<q<p(\nu_*)\) proves openness.
\end{proof}

\subsection{Differential estimates for the coordinate}

\begin{lemma}[Differentials of the coordinate map]
\label{lem:F-calculus}
The pointwise map \(F\) in \eqref{eq:intro-coordinate} is real analytic on \(\C\).
Writing \(s=\sqrt{1+|\xi|^2}\), we have
\begin{equation}\label{eq:DF-formula}
    DF_\xi(h)=\frac{h}{1+s}
      -\frac{\xi\operatorname{Re}(\overline\xi h)}{s(1+s)^2},
    \qquad DF_0=\tfrac12 I,\qquad D^2F_0=0.
\end{equation}
In particular,
\begin{align}
    \|DF_\xi\|_{\op} &\le \tfrac12,\label{eq:DF-bounded}\\
    \frac{\|DF_\xi\|_{\op}}{1-|F(\xi)|^2}&=\tfrac12.
    \label{eq:F-cancellation}
\end{align}
For each integer \(j\ge1\) there are constants \(c_j,c'_j\) such that
\begin{equation}\label{eq:F-higher-cancellation}
    \|D^jF_\xi\|_{\op}\le c_j(1+|\xi|)^{-j},
    \qquad
    \frac{\|D^jF_\xi\|_{\op}}{1-|F(\xi)|^2}\le c'_j.
\end{equation}
Consequently \(DF\) is globally Lipschitz and
\begin{equation}\label{eq:F-quadratic}
    |F(\xi+h)-F(\xi)-DF_\xi(h)|\le C|h|^2.
\end{equation}
\end{lemma}

\begin{proof}
Differentiation gives \eqref{eq:DF-formula}, and oddness gives
\(D^2F_0=0\). The radial and tangential eigenvalues of \(DF_\xi\) are
\(1/[s(1+s)]\) and \(1/(1+s)\), respectively. The latter is the
operator norm and equals \((1-|F(\xi)|^2)/2\).
For \(L=|\xi|\ge1\), write
\[
    F(Lu)=\frac{u}{L^{-1}+\sqrt{L^{-2}+|u|^2}}.
\]
The right-hand side is smooth in \((L^{-1},u)\) on
\([0,1]\times\{1/2\le|u|\le2\}\), so its derivatives in \(u\) are
uniformly bounded there. Rescaling at \(u=\xi/L\) gives
\(\|D^jF_\xi\|_{\op}\le C_j L^{-j}\). On the unit disk the derivatives
are bounded by compactness. This proves \eqref{eq:F-higher-cancellation}
since \((1-|F(\xi)|^2)^{-1}=(1+s)/2\). Taylor's formula with \(j=2\)
gives \eqref{eq:F-quadratic}.
\end{proof}

\subsection{A weighted resolvent estimate}
\label{subsec:quantitative-AIM}

Fix \(B_{R_0}=\{z:|z|<R_0\}\) containing \(\overline{\mathcal U}\).
Let \(\mathcal A\) be a family of measurable Beltrami coefficients vanishing
off \(\mathcal U\), and suppose that for some \(q>p_{\mathrm{res}}\),
\begin{equation}\label{eq:uniform-moment-majorant}
   \sup_{\mu\in\mathcal A}\int_{\mathcal U}e^{qK_\mu}\dA\le M_0.
\end{equation}
Choose \(p_{\mathrm{res}}<\alpha<q\).
The weight in \cref{thm:prelim-aim} can be chosen for each \(\mu\)
so that its local exponential moments are uniform over \(\mathcal A\).

\begin{lemma}\label{lem:uniform-AIM-majorants}
For \(\mu\in\mathcal A\), set
\(k(\zeta)=K_\mu(R_0\zeta)\chi_{\D}(\zeta)\). With \(k^*_\alpha\)
as in \eqref{eq:aim-majorant}, put
\begin{equation}\label{eq:kappa-mu}
   \kappa_\mu(z)=k^*_\alpha(z/R_0).
\end{equation}
Then \(\kappa_\mu\ge K_\mu\chi_{B_{R_0}}\) almost everywhere on \(\C\),
and \([\kappa_\mu]_{\BMO}\le C/\alpha\) for a universal \(C\).
For each \(0<\sigma<\alpha\),
\begin{equation}\label{eq:kappa-uniform-exp}
   \sup_{\mu\in\mathcal A}\int_{B_{R_0}}e^{\sigma\kappa_\mu}\dA
   \le C(R_0,q,\alpha,\sigma,M_0).
\end{equation}
\end{lemma}

\begin{proof}
The first assertions follow from \cref{thm:prelim-aim} and the
invariance of the BMO seminorm under dilations. Set \(s=\sigma/\alpha<1\)
and \(G=G_{\alpha,k}\). On \(\D\) we use \(e^{\alpha k}\le e^{qk}\),
\eqref{eq:uniform-moment-majorant}, and \(K_\mu=1\) off \(\mathcal U\).
On \(\C\setminus\D\) we have \(G(\zeta)=(1+|\zeta|)^{-3}\).
Together these give a uniform bound for \(\|G\|_1\). The weak \((1,1)\) estimate for \(\HL\)
on the finite-area disk gives
\[
   \int_{B_{R_0}}e^{\sigma\kappa_\mu}\dA
   \le R_0^2\,2^{3s}\int_\D(\HL G)^s\dA
   \le C(R_0,s)\|G\|_1^s .
\]
\end{proof}

A common pointwise distortion bound with a finite \(\alpha\)-exponential
moment would give a common majorant by monotonicity of \(\HL\).
Even a Luxemburg ball about zero of positive radius admits no finite
almost-everywhere distortion envelope on \(\mathcal U\). Indeed, any
candidate is bounded by some \(M\) on a set of positive area. Fix
\(t>M\) and take subsets \(A\) of this set of arbitrarily small positive
area. On \(A\) the distortion of \(F(t\chi_A)\) exceeds \(M\), while
\begin{equation}\label{eq:indicator-norm}
    \|t\chi_A\|_\Phi=\frac{t}{\Phi^{-1}(1/|A|)}\longrightarrow0
    \qquad(|A|\to0).
\end{equation}

\begin{proposition}[Principal resolvent]\label{prop:uniform-principal-resolvent}
For \(\mu\in\mathcal A\), let \(H\) vanish off \(\mathcal U\) and
satisfy \(\kappa_\mu H\in L^2(\C)\). The equation
\(\omega-\mu\Beurling\omega=H\) has a unique solution in \(L^2(\C)\)
with \(\|\omega\|_2\le4\|\kappa_\mu H\|_2\).
Consequently \(w=\Cauchy\omega\) is the unique solution in
\(W^{1,2}_{\loc}(\C)\) to \(\dbar w-\mu\dz w=H\) with \(w(z)=O(1/z)\),
and
\begin{equation}\label{eq:principal-resolvent-estimate}
    \|\dbar w\|_2+\|\dz w\|_2\le8\|\kappa_\mu H\|_2 .
\end{equation}
\end{proposition}

\begin{proof}
Apply \cref{thm:prelim-aim} to \(\mu(R_0\,\cdot)\) and
\(H(R_0\,\cdot)\) on the unit disk. The Beurling transform commutes
with dilations. Returning to \(z=R_0\zeta\) multiplies both \(L^2\)
norms by the same factor, so \eqref{eq:AIM-resolvent-4} gives the stated
bound with \(\kappa_\mu\) from \eqref{eq:kappa-mu}.

The equation gives \(\omega=0\) off \(\mathcal U\), hence
\(w=\Cauchy\omega=O(1/z)\) at infinity. The identities
\(\dbar w=\omega\) and \(\dz w=\Beurling\omega\) prove
\eqref{eq:principal-resolvent-estimate}. Conversely, if \(v\) is a
solution with the stated regularity and normalization, then
\(\dbar v\) is compactly supported and belongs to \(L^2(\C)\).
The entire function \(v-\Cauchy(\dbar v)\) vanishes at infinity, so
\(v=\Cauchy(\dbar v)\). The uniqueness assertion of
\cref{thm:prelim-aim} now applies to \(\dbar v\).
\end{proof}

\subsection{Jacobian bounds and weighted integrals}

For \(\nu\in\mathscr D_\Lambda\), denote the principal solution from
\cref{thm:prelim-david} by \(\Sol(\nu)=f_\nu=f^{\mu_\nu}\).

\begin{proposition}[Uniform Jacobian bounds]\label{prop:analytic-package}
Assume \eqref{eq:uniform-moment-majorant} with
\(q>\max\{1,p_{\mathrm{res}}\}\). Each \(f^\mu\) belongs to
\(W^{1,2}_{\loc}(\C)\), and both \(f^\mu\) and its inverse preserve
sets of area zero. For each disk \(B_R\supset\overline{\mathcal U}\)
and each \(0<b<q\),
\begin{equation}\label{eq:uniform-J-log-power}
 \sup_{\mu\in\mathcal A}\int_{B_R}
   J_{f^\mu}\log^b(e+J_{f^\mu})\dA<\infty .
\end{equation}
\end{proposition}

\begin{proof}
\Cref{thm:prelim-agrs} gives the Sobolev regularity and the null-set
properties. Choose \(b<s<q\). After dilation of \(B_R\) to \(\D\),
\eqref{eq:uniform-moment-majorant} and \(K_\mu=1\) off \(\mathcal U\)
give a uniform bound for the moment in \cref{thm:prelim-agrs}.
Consequently \eqref{eq:agrs-area} gives, uniformly
in \(\mu\),
\[
  \int_E J_{f^\mu}\dA=|f^\mu(E)|
   \le C\bigl[\log(e+|B_R|/|E|)\bigr]^{-s}
   \qquad(E\subset B_R).
\]
For \(E=B_R\) this bounds \(\int_{B_R}J_{f^\mu}\dA\).
Consequently \(E_t=\{z\in B_R:J_{f^\mu}(z)>t\}\) has
\(|E_t|\le C/t\). The same area estimate now gives
\(\int_{E_t}J_{f^\mu}\dA\le C'\log^{-s}(e+t)\).
Layer-cake integration against the derivative of
\(t\mapsto\log^b(e+t)\) proves \eqref{eq:uniform-J-log-power}, since
\(\int_1^\infty (e+t)^{-1}\log^{b-1-s}(e+t)\,dt<\infty\).
\end{proof}

\begin{lemma}[Exponential factors against a Jacobian weight]
\label{lem:exponential-product}
Let \(E\subset\C\) have finite area and let \(b>0\). Suppose that
\(J\ge0\) and \(b_1,\ldots,b_N\ge0\) satisfy
\[
    \int_E J\log^b(e+J)\dA\le M_J,
    \qquad
    \int_E e^{\sigma_j b_j}\dA\le M_j\quad(1\le j\le N)
\]
for some \(\sigma_j>0\).
If \(a_j,s_k\ge0\) and \(A=\sum_j a_j+\sum_{k=1}^m s_k\le b\), then
\begin{equation}\label{eq:product-estimate}
    \int_E J\prod_{j=1}^N b_j^{a_j}\prod_{k=1}^m|g_k|^{s_k}\dA
    \le C\prod_{k=1}^m\|g_k\|_{L^\Phi(E)}^{s_k},
    \qquad g_k\in L^\Phi(E).
\end{equation}
Here \(C\) depends only on \(|E|\), the exponents and the indicated
moment bounds.
\end{lemma}

\begin{proof}
Omit factors of degree zero. The case \(A=0\) is immediate, as is the
case in which a remaining \(g_k\) vanishes almost everywhere. By homogeneity
assume \(\|g_k\|_{L^\Phi(E)}=1\). Let \(S\) be the pointwise maximum of
the \(b_j\) and \(|g_k|\), and choose \(0<\varepsilon<1\) with
\(\varepsilon\le\sigma_j\) for all \(j\). The product of the factors is at
most \(S^A\), and \(\int_E e^{\varepsilon S}\dA\) is uniformly bounded
by the hypotheses and \cref{lem:luxemburg-exponential}.
On \(\{S\le 2\varepsilon^{-1}\log(e+J)\}\), we have
\(JS^A\le C J\log^b(e+J)\).
On the complementary set \(J\le e^{\varepsilon S/2}\) and
\(JS^A\le C e^{\varepsilon S}\). Integration proves the estimate.
\end{proof}

\begin{lemma}[Principal Poincar\'e inequality]\label{lem:principal-poincare}
If \(v\in W^{1,2}_{\loc}(\C)\) is holomorphic off \(\overline{B_R}\)
and \(v(z)=O(1/z)\) at infinity, then
\begin{equation}\label{eq:principal-poincare}
   \|v\|_{L^2(B_{2R})}\le C_R\|Dv\|_{L^2(B_{2R})}.
\end{equation}
\end{lemma}

\begin{proof}
The Laurent expansion at infinity gives \(\int_Av\dA=0\) on the
annulus \(A=B_{2R}\setminus\overline{B_R}\). If \(m\) is the average
of \(v\) on \(B_{2R}\), then
\(|m|\le|A|^{-1/2}\|v-m\|_{L^2(A)}\). Apply the usual Poincar\'e
inequality to \(v-m\) on \(B_{2R}\).
\end{proof}

\section{First variation on the high-moment domain}\label{sec:frechet}

\begin{mainthm}\label{thm:intro-main}
The principal solution map
\[
    \Sol\colon\mathscr D_\Lambda\longrightarrow W^{1,2}_{\loc}(\C),
    \qquad \nu\longmapsto f_\nu=f^{F(\nu)},
\]
is locally real \(C^{1,1}\) in each local Sobolev norm. For
\(\eta\in L^\Phi_{\mathcal U}(\C)\), the derivative \(V_\nu[\eta]=D\Sol_\nu[\eta]\)
is the unique solution \(V\in W^{1,2}_{\loc}(\C)\) of
\begin{equation}\label{eq:linearized-equation}
   (\dbar-F(\nu)\dz)V=DF_\nu(\eta)\,\dz f_\nu,
   \qquad V(z)=O(1/z)\quad(z\to\infty).
\end{equation}
Here \(DF_\nu(\eta)\) denotes the pointwise real differential.

More precisely, given \(\nu_*\in\mathscr D_\Lambda\) and a bounded
disk \(B\subset\C\), there are \(r,C>0\) such that, on the Luxemburg ball
\(\mathcal V=B_\Phi(\nu_*,r)\subset\mathscr D_\Lambda\),
\begin{align}
   \|V_\nu[\eta]\|_{W^{1,2}(B)}
      &\le C\|\eta\|_\Phi
      &&(\nu\in\mathcal V),                                      \label{eq:uniform-variation-bound}\\
   \|f_{\nu+h}-f_\nu-V_\nu[h]\|_{W^{1,2}(B)}
      &\le C\|h\|_\Phi^2
      &&(\nu,\nu+h\in\mathcal V).                                \label{eq:uniform-quadratic-remainder}
\end{align}
On the concentric ball of radius \(r/4\),
\begin{equation}\label{eq:intro-lipschitz-derivative}
    \|D\Sol_y-D\Sol_x\|_{\mathcal L_{\R}
       (L^\Phi_{\mathcal U}(\C),W^{1,2}(B))}
    \le 6C\|y-x\|_\Phi .
\end{equation}
\end{mainthm}

\subsection{Solution of the linearized equation}\label{sec:pullback}

Fix \(\nu\in\mathscr D_\Lambda\) and \(\eta\in L^\Phi_{\mathcal U}(\C)\).
Put \(\mu=F(\nu)\), \(f=f^\mu\), \(K=K_\mu\), \(J=J_f\), and
\(A_\nu[\eta]=DF_\nu(\eta)\).
To solve \eqref{eq:linearized-equation}, define the source on the target plane by
\begin{equation}\label{eq:rho-definition}
    \rho_{\nu,\eta}(f(z))
       =\frac{A_\nu[\eta](z)}{1-|\mu(z)|^2}
        \frac{f_z(z)}{\overline{f_z(z)}}.
\end{equation}
We set the last factor equal to zero when \(f_z=0\). The image of this
set is null by the area formula, since \(J=(1-|\mu|^2)|f_z|^2\).
Property \(N\) ensures that \(\rho_{\nu,\eta}\) vanishes almost everywhere
off the bounded set \(f(\mathcal U)\).

\begin{lemma}[Estimate for the transformed source]
\label{lem:rho-L7}
Locally uniformly in \(\nu\in\mathscr D_\Lambda\),
\begin{equation}\label{eq:rho-L7-estimate}
    \|\rho_{\nu,\eta}\|_{L^7(\C)}\le C\|\eta\|_\Phi.
\end{equation}
\end{lemma}

\begin{proof}
By \eqref{eq:F-cancellation}, \(|\rho_{\nu,\eta}\circ f|\le|\eta|/2\)
almost everywhere. The area formula gives
\[
    \|\rho_{\nu,\eta}\|_7^7
       \le 2^{-7}\int_{\mathcal U}|\eta|^7 J\dA
       \le C\|\eta\|_\Phi^7.
\]
The last inequality follows from \cref{lem:exponential-product}
with \(b=8\). Its constant is locally uniform by
\cref{lem:small-ball-moment,prop:analytic-package} and
\eqref{eq:distortion-comparison}.
\end{proof}

Set
\begin{equation}\label{eq:W-and-V}
    W=\Cauchy\rho_{\nu,\eta},\qquad V_\nu[\eta]=W\circ f .
\end{equation}
The Cauchy transform has the continuous \(W^{1,7}_{\loc}\) representative
satisfying \(W(\zeta)=O(1/\zeta)\) at infinity. Since
\(\dbar W=\rho_{\nu,\eta}\) and \(\dz W=\Beurling\rho_{\nu,\eta}\),
boundedness of the Beurling transform gives
\begin{equation}\label{eq:DW-L7}
    \|DW\|_{L^7(\C)}\le C\|\eta\|_\Phi.
\end{equation}

\begin{proposition}[Pullback formula and weighted estimate]
\label{prop:pullback-variation}
The map \(\eta\mapsto V_\nu[\eta]\) is real linear and gives the unique
\(W^{1,2}_{\loc}(\C)\) solution of \eqref{eq:linearized-equation} with
\(V(z)=O(1/z)\) at infinity. For each bounded disk \(B\),
\begin{equation}\label{eq:pullback-unweighted}
    \|DV_\nu[\eta]\|_{L^2(\C)}
       +\|V_\nu[\eta]\|_{W^{1,2}(B)}\le C_B\|\eta\|_\Phi.
\end{equation}
For every measurable \(\kappa\ge0\)
on \(\mathcal U\) with \(\int_{\mathcal U}e^{\sigma\kappa}\dA\le M\)
for some \(\sigma>0\), one has
\begin{equation}\label{eq:weighted-pullback}
    \|\kappa g\,DV_\nu[\eta]\|_{L^2(\mathcal U)}
       \le C\|g\|_\Phi\|\eta\|_\Phi,
    \qquad g\in L^\Phi_{\mathcal U}(\C).
\end{equation}
For fixed \(\sigma,M\), the constant is locally uniform in \(\nu\).
\end{proposition}

\begin{proof}
The complex chain rule and \(f_{\bar z}=\mu f_z\) give, for smooth \(U\),
\begin{equation}\label{eq:pullback-identity}
    (\dbar-\mu\dz)(U\circ f)
       =(1-|\mu|^2)\overline{f_z}\,((\dbar U)\circ f).
\end{equation}
For a disk \(B\) and \(p>2\), set \(r=p/(p-2)\). The identity
\(|Df|_{\op}^2=KJ\), change of variables, and H\"older's inequality give
\begin{equation}\label{eq:pullback-Lp-L2}
    \|D(U\circ f)\|_{L^2(B)}^2
       \le \|K\circ f^{-1}\|_{L^r(f(B))}
           \|DU\|_{L^p(f(B))}^2,
\end{equation}
where matrix norms on the two derivatives of \(U\) and \(U\circ f\)
are Hilbert--Schmidt norms. For \(p=7\), the \(r\)-th power of the
weight norm equals \(\int_B K^{7/5}J\dA<\infty\) by
\cref{lem:exponential-product,prop:analytic-package}. Choose a disk
containing \(f(\overline B)\) in its interior and approximate \(W\) there
in \(W^{1,7}\) by smooth functions. Morrey's inequality gives uniform
convergence on \(f(\overline B)\), and the estimate underlying
\eqref{eq:pullback-Lp-L2} gives convergence of their pulled-back
gradients to \((DW\circ f)Df\) in \(L^2(B)\).
Thus \(V=W\circ f\) belongs to \(W^{1,2}_{\loc}\), with the same chain
rule. Substitution of \(\dbar W=\rho_{\nu,\eta}\) into
\eqref{eq:pullback-identity} proves \eqref{eq:linearized-equation}.
The asymptotics of \(W\) and \(f\) give \(V(z)=O(1/z)\).

For \eqref{eq:weighted-pullback}, put
\(Q=(\kappa^2|g|^2K)\circ f^{-1}\) on \(f(\mathcal U)\).
The same change of variables gives
\[
    \|\kappa g\,DV\|_{L^2(\mathcal U)}^2
       \le \int_{f(\mathcal U)}Q|DW|^2\dA
       \le \|Q\|_{7/5}\|DW\|_7^2.
\]
By \cref{lem:exponential-product},
\[
    \|Q\|_{7/5}^{7/5}
       =\int_{\mathcal U}
           \kappa^{14/5}|g|^{14/5}K^{7/5}J\dA
       \le C\|g\|_\Phi^{14/5},
\]
since the three powers sum to seven. Combining this with
\eqref{eq:DW-L7} proves the estimate with the stated uniformity.

Finally, for the majorant \(\kappa_\mu\) of
\cref{lem:uniform-AIM-majorants}, \eqref{eq:DF-bounded} and
\(|f_z|^2\le KJ\) give the locally uniform estimate
\begin{equation}\label{eq:pullback-source-bound}
    \|\kappa_\mu A_\nu[\eta] f_z\|_2^2
       \le \tfrac14\int_{\mathcal U}\kappa_\mu^2|\eta|^2 KJ\dA
       \le C\|\eta\|_\Phi^2.
\end{equation}
Here the total degree in \cref{lem:exponential-product} is five.
Every normalized solution is holomorphic off the fixed support and has
square-integrable gradient near infinity by its Laurent expansion.
Thus \cref{prop:uniform-principal-resolvent} gives both uniqueness and
the global gradient bound in \eqref{eq:pullback-unweighted}.
The local norm estimate follows from \cref{lem:principal-poincare} on
a disk containing \(\overline B\cup\overline{\mathcal U}\).
\end{proof}

\subsection{The quadratic remainder}

\begin{proof}[Proof of \cref{thm:intro-main}]
Choose \(q\) with \(\max\{p_{\mathrm{res}},8\}<q<p(\nu_*)/2\).
By \cref{lem:small-ball-moment} and \eqref{eq:distortion-comparison},
the distortions \(K_{F(\xi)}\), \(\xi\in\mathcal V\), have a common
\(q\)-exponential moment bound on \(\mathcal U\) once \(r\) is small
enough. Fix \(p_{\mathrm{res}}<\alpha<q\) and \(0<\sigma<\alpha\).
By \cref{prop:analytic-package,lem:uniform-AIM-majorants}, the
Jacobians have a uniform \(L\log^8 L\) bound on each fixed disk, and
the corresponding weights \(\kappa_{F(\xi)}\) have a uniform
\(\sigma\)-exponential moment on \(\mathcal U\).
The unweighted estimate \eqref{eq:pullback-unweighted} now gives
\eqref{eq:uniform-variation-bound} for every direction \(\eta\).

Write \(\mu=F(\nu)\), \(\mu_h=F(\nu+h)\), \(f=f_\nu\) and \(V=V_\nu[h]\).
Subtracting the equations for \(f_{\nu+h}\), \(f\) and \(V\), with
\(\mu_h\) on the left, gives the equation for the remainder
\(Z=f_{\nu+h}-f-V\),
\begin{equation}\label{eq:Z-equation}
  (\dbar-\mu_h\dz)Z
   =\underbrace{\bigl(\mu_h-\mu-DF_\nu(h)\bigr)f_z}_{R_1}
    +\underbrace{(\mu_h-\mu)V_z}_{R_2}.
\end{equation}
Both terms vanish off \(\mathcal U\), and \(Z(z)=O(1/z)\) at infinity.
Set \(\kappa_h=\kappa_{\mu_h}\). By \eqref{eq:F-quadratic} and
\(|f_z|^2\le K_\mu J_f\),
\[
    \|\kappa_h R_1\|_2^2
    \le C\int_{\mathcal U}\kappa_h^2|h|^4K_\mu J_f\dA
    \le C\|h\|_\Phi^4 .
\]
The last inequality is \cref{lem:exponential-product} with total
degree \(2+4+1=7\). It uses separate moment bounds for \(\kappa_h\) and
\(K_\mu\), which belong to different coefficients. For the second term,
\eqref{eq:DF-bounded} gives \(|\mu_h-\mu|\le |h|/2\).
Use \eqref{eq:weighted-pullback} at the base point \(\nu\), with
weight \(\kappa_h\) and both directions equal to \(h\), to obtain
\[
    \|\kappa_h R_2\|_2
    \le \tfrac12\|\kappa_h |h|\,DV_\nu[h]\|_2
    \le C\|h\|_\Phi^2 .
\]
Now \cref{prop:uniform-principal-resolvent} applied to
\eqref{eq:Z-equation} bounds \(\|DZ\|_{L^2(\C)}\) by \(C\|h\|_\Phi^2\).
Indeed, \(Z\) is locally in \(W^{1,2}\) and holomorphic near infinity
with derivative \(O(1/z^2)\), so it belongs to the uniqueness class of
the resolvent. Then \cref{lem:principal-poincare} gives
\eqref{eq:uniform-quadratic-remainder}. Since \(V_\nu\) is real-linear
by \cref{prop:pullback-variation}, the two estimates identify it as the
Fr\'echet derivative.

For the Lipschitz estimate, write \(G(\xi)=f_\xi|_B\) and
\(E_x(u)=G(x+u)-G(x)-DG_x[u]\). Fix distinct \(x,y\in B_\Phi(\nu_*,r/4)\),
set \(d=y-x\), \(t=\|d\|_\Phi\), and let \(\|v\|_\Phi=1\).
All three points \(x,y,y+tv\) lie in \(\mathcal V\), and
\[
    t(DG_y-DG_x)[v]=E_x(d+tv)-E_x(d)-E_y(tv).
\]
The bound \(\|E_x(u)\|_{W^{1,2}(B)}\le C\|u\|_\Phi^2\)
gives \(\|(DG_y-DG_x)[v]\|_{W^{1,2}(B)}\le6Ct\).
Taking the supremum over \(v\) proves \eqref{eq:intro-lipschitz-derivative}.
\end{proof}

At the origin \(f_0=\id\) and \(DF_0=I/2\), so the principal solution of
\eqref{eq:linearized-equation} and its quadratic remainder are
\begin{equation}\label{eq:origin-derivative}
    D\Sol_0[\eta]=\tfrac12\Cauchy\eta,\qquad
    \|f_h-\id-\tfrac12\Cauchy h\|_{W^{1,2}(B)}\le C_B\|h\|_\Phi^2
\end{equation}
for every direction \(\eta\) and sufficiently small \(h\).

The moment threshold in \cref{thm:intro-main} permits the use
of the principal resolvent and supplies the
logarithmic Jacobian moment in the product estimates. For the
\(W^{1,2}_{\loc}\) regularity of each individual solution,
\(e^{pK_\mu}\in L^1(\mathcal U)\) for some \(p>1\) suffices, and the
radial examples in \cite[Corollary~1.2 and (3)]{AGRS} show the sharpness
of this latter threshold. The argument leaves open whether this threshold also suffices for local \(C^{1,1}\) regularity.

\section{Smoothness on the Orlicz heart}\label{sec:smooth-heart}

Write \(E=E^\Phi_{\mathcal U}(\C)\) and
\(X_\infty=L^\infty_{\mathcal U}(\C)\).  At a point of \(E\), the
distortion has exponential moments of every finite order.  This allows
the exponent in the following estimate to be chosen according to the
order of differentiation.

\begin{lemma}[Transport of directions]\label{lem:LP-transport}
Fix \(\nu_*\in E\), \(2\le P<\infty\), and a disk \(B\) containing
\(\overline{\mathcal U}\).  There exist \(r,C>0\) such that, for
\(\nu\in B_E(\nu_*,r)\), \(f=\Sol(\nu)\), and \(K=K_{F(\nu)}\),
\begin{align}
  \|h\circ f^{-1}\|_{L^P(f(B))}
     &\le C\|h\|_\Phi
       \qquad(h\in L^\Phi_{\mathcal U}(\C)),\label{eq:LP-h-transport}\\
  |f(B)|+\|K\circ f^{-1}\|_{L^2(f(B))}
     &\le C.\label{eq:LP-K-transport}
\end{align}
\end{lemma}

\begin{proof}
Choose \(q>\max\{P,8,p_{\mathrm{res}}\}\).  The heart condition,
\(K\le1+2|\nu|\), and \cref{lem:small-ball-moment} give a ball on
which \(\int_B e^{qK}\dA\) is uniformly bounded.  By
\cref{prop:analytic-package}, so is
\(\int_B J_f\log^P(e+J_f)\dA\).  The area formula gives
\[
  \int_{f(B)}|h\circ f^{-1}|^P\dA=\int_B |h|^P J_f\dA,
  \qquad
  \int_{f(B)}|K\circ f^{-1}|^2\dA=\int_B K^2J_f\dA .
\]
Apply \cref{lem:exponential-product} with degree \(P\) to the first
integral and with degree \(2\) to the second.  Finally,
\(|f(B)|=\int_B J_f\dA\) has the same uniform bound.
\end{proof}

\begin{theorem}[Higher variations on the heart]\label{thm:intro-smooth}
The map \(\Sol:E\to W^{1,2}_{\loc}(\C)\) is real Fr\'echet
\(C^\infty\).  Given \(\nu_*\in E\), an integer \(n\ge1\), and a
disk \(B_R\), there are \(r,C>0\) such that for \(\nu\in B_E(\nu_*,r)\),
\(1\le j\le n\), and \(h_1,\ldots,h_j\in E\),
\begin{equation}\label{eq:intro-higher-bound}
  \|D^j\Sol_\nu[h_1,\ldots,h_j]\|_{W^{1,2}(B_R)}
    \le C\prod_{i=1}^j\|h_i\|_\Phi .
\end{equation}
At the origin the second derivative is
\[
  D^2\Sol_0[h,k]=\tfrac14\Cauchy\bigl(h\,\Beurling k+k\,\Beurling h\bigr).
\]
\end{theorem}

\begin{proof}
Enlarge \(B_R\) to contain \(\overline{\mathcal U}\), and set \(m=n+1\).
Take \(P=4m\) and let \(U\) be the ball supplied by
\cref{lem:LP-transport} for \(B=B_{2R}\).

First let \(\nu\in U\cap X_\infty\) and
\(h_1,\ldots,h_m\in X_\infty\).  Write \(f=\Sol(\nu)\) and \(\mu=F(\nu)\).
In the uniformly elliptic case \cite{AB}, the bound
\(\|\mu\|_\infty<1\) and the \(L^2\)-isometry of \(\Beurling\) give the
principal solution in the form
\[
  f(z)=z+\Cauchy\omega(z),\qquad
  \omega=(I-\mu\Beurling)^{-1}\mu .
\]
With the \(L^\infty\) norm on \(X_\infty\), pointwise substitution by \(F\) is
real \(C^\infty\) on bounded sets, and inversion is analytic on the group of invertible
bounded operators on \(L^2\).  Thus all mixed variations
\(V_I=D^{|I|}(\Sol|_{X_\infty})_\nu[h_i:i\in I]\), for
\(\varnothing\ne I\subseteq\{1,\ldots,m\}\), are Cauchy transforms of
\(L^2\) functions supported in \(\overline{\mathcal U}\).  In particular,
\(DV_I\in L^2(\C)\), and \(V_I\) is holomorphic off
\(\overline{\mathcal U}\) with \(V_I(z)=O(1/z)\) at infinity.

Set \(V_\varnothing=f\), \(W_I=V_I\circ f^{-1}\), and
\(F_J=D^{|J|}F_\nu[h_i:i\in J]\), with \(F_J\) evaluated pointwise.
Differentiation of the Beltrami equation gives
\begin{equation}\label{eq:higher-var-eq}
  (\dbar-\mu\dz)V_I
     =\sum_{\varnothing\ne J\subseteq I} F_J\,\dz V_{I\setminus J}.
\end{equation}
For \(I\ne\varnothing\), quasiconformal change of variables gives
\(DW_I\in L^2(\C)\), hence \(\dz W_I=\Beurling(\dbar W_I)\).
Applying \eqref{eq:pullback-identity} to \eqref{eq:higher-var-eq} gives,
at \(y=f(z)\) and with \(\theta=f_z/\overline{f_z}\),
\begin{equation}\label{eq:higher-target-eq}
  \dbar W_I(y)
   =\sum_{\varnothing\ne J\subseteq I}
      \frac{F_J(z)}{1-|\mu(z)|^2}
      \left(\theta(z)\,\dz W_{I\setminus J}(y)
           +\overline{\mu(z)}\,\dbar W_{I\setminus J}(y)\right).
\end{equation}
By \eqref{eq:F-higher-cancellation}, \eqref{eq:LP-h-transport}, and
H\"older's inequality,
\[
  \left\|\left(\frac{F_J}{1-|\mu|^2}\right)\circ f^{-1}
        \right\|_{L^{P/|J|}(\C)}
     \le C_m\prod_{i\in J}\|h_i\|_\Phi .
\]
We claim that
\begin{equation}\label{eq:DW-I-bound}
  \|DW_I\|_{L^{P/|I|}(\C)}\le C_m\prod_{i\in I}\|h_i\|_\Phi .
\end{equation}
Indeed, the term \(J=I\) in \eqref{eq:higher-target-eq} uses only
\(W_\varnothing=\id\).  For \(J\subsetneq I\), the inductive bound for
\(DW_{I\setminus J}\) and H\"older's inequality give the same estimate
for that summand since \(|I|/P=|J|/P+|I\setminus J|/P\).  This bounds
\(\dbar W_I\) in \(L^{P/|I|}\), and boundedness of \(\Beurling\)
on this space gives \eqref{eq:DW-I-bound}.

Write \(p=P/|I|\ge4\).  By \eqref{eq:pullback-Lp-L2}, the local
\(L^2\) norm of \(DV_I\) on \(B_{2R}\) is at most a constant times
\[
   \|K\circ f^{-1}\|_{L^{p/(p-2)}(f(B_{2R}))}^{1/2}
   \|DW_I\|_{L^p(\C)}.
\]
Since \(p/(p-2)\le2\), \eqref{eq:LP-K-transport} controls the first
factor.  The principal Poincar\'e inequality
(\cref{lem:principal-poincare}) therefore gives
\begin{equation}\label{eq:bounded-higher-bound}
   \|V_I\|_{W^{1,2}(B_R)}\le C_m\prod_{i\in I}\|h_i\|_\Phi ,
   \qquad 1\le |I|\le m .
\end{equation}
The constants are independent of the \(L^\infty\) norms of the base
point and the directions.

It remains to identify the derivatives on \(E\).  For \(a\in U\cap
X_\infty\), \eqref{eq:bounded-higher-bound} and the density of
\(X_\infty\) extend its \(j\)-th variation to a bounded \(j\)-linear
map \(G_j(a):E^j\to W^{1,2}(B_R)\), \(1\le j\le m\).  Put
\(G_0(a)=\Sol(a)|_{B_R}\).  For \(a,b\in U\cap X_\infty\), integration
of the variation of order \(j+1\) along \([a,b]\), first in bounded
directions and then by density, gives
\[
   \|G_j(b)-G_j(a)\|_{\op}\le C_m\|b-a\|_\Phi,\qquad 0\le j\le n .
\]
For \(j=0\) the norm on the left is the \(W^{1,2}(B_R)\) norm.
Consequently \(G_0,\ldots,G_n\) extend continuously to \(U\);
by \cref{thm:intro-main} the extension of \(G_0\) is \(\Sol|_{B_R}\).
Approximating both endpoints by bounded data in the convex ball \(U\)
and using these Lipschitz bounds yields
\[
  G_j(b)-G_j(a)=
     \int_0^1 G_{j+1}\bigl(a+t(b-a)\bigr)[b-a,\cdot]\,dt,
     \qquad a,b\in U,\quad 0\le j<n .
\]
Hence \(DG_j=G_{j+1}\), proving \(C^n\) regularity and
\eqref{eq:intro-higher-bound}.  The disk and \(n\) were arbitrary.

Finally \(DF_0=\tfrac12 I\) and \(D^2F_0=0\).  For bounded \(h,k\),
\eqref{eq:higher-var-eq} and the first variation at zero give
\(\dbar D^2\Sol_0[h,k]=\tfrac14(h\,\Beurling k+k\,\Beurling h)\).
The embedding \(E\hookrightarrow L^4(\mathcal U)\) and boundedness of
\(\Beurling\) on \(L^4\) make this right-hand side continuous in
\((h,k)\) with values in \(L^2\), with support in
\(\overline{\mathcal U}\).  Applying the principal Cauchy transform
and passing to the limit in bounded directions proves the formula for
all \(h,k\in E\).
\end{proof}

\begin{proposition}[Directional Taylor radius at the origin]
\label{prop:directional-radius}
Let \(h\in E\) and let \(B\) be a bounded disk containing
\(\overline{\mathcal U}\). Write \(R_B(h)\) for the radius of convergence
in \(W^{1,2}(B)\) of the Taylor series at zero of the real curve
\(t\mapsto\Sol(th)|_B\). Then
\begin{equation}\label{eq:directional-radius}
    R_B(h)=\frac{1}{\|h\|_\infty},
    \qquad 1/0=\infty,\quad 1/\infty=0 .
\end{equation}
The curve is real analytic at zero exactly when \(h\in L^\infty\).
In particular, \(\Sol\) on \(E\) is not real analytic at zero.
\end{proposition}

\begin{proof}
Put \(f_t=\Sol(th)\) for real \(t\), and denote its Taylor coefficients
on \(B\) by \(a_n=(1/n!)\,\frac{d^n}{dt^n}f_t|_{t=0}\).
They exist by \cref{thm:intro-smooth}, with \(a_0(z)=z\).
To bound the radius from above, set \(P_t=\dz f_t\) and
\(Q_t=\dbar f_t\). The Beltrami equation gives
\[
    Q_t=\frac{th}{1+\sqrt{1+t^2|h|^2}}P_t,\qquad
    (P_t+t\overline h\,Q_t)^2=(1+t^2|h|^2)P_t^2 .
\]
On \(B_N=B\cap\{|h|\le N\}\) the second equality is an identity of
\(C^\infty\) curves with values in \(L^1(B_N)\), since multiplication
maps \(L^2\times L^2\) continuously into \(L^1\). Taking each Taylor
coefficient therefore gives the corresponding product identity in
\(L^1(B_N)\). Taking all orders and all integers \(N\ge1\) yields the
formal identity almost everywhere on \(B\).

Put \(R=R_B(h)\) and suppose \(R>0\). Fix \(r\) with \(0<r<R\).
Absolute convergence gives \(\sum_n r^n\|a_n\|_{W^{1,2}(B)}<\infty\). Hence
\[
    \sum_{n\ge0}r^n\bigl(|\dz a_n|+|\dbar a_n|\bigr)\in L^2(B).
\]
Taking a countable sequence of radii tending to \(R\), we obtain, for
almost every \(z\in B\), two holomorphic functions on \(|\tau|<R\),
\[
    P(\tau,z)=\sum_{n\ge0}\tau^n\dz a_n(z),\qquad
    Q(\tau,z)=\sum_{n\ge0}\tau^n\dbar a_n(z).
\]
The coefficient identity just proved gives throughout this disk
\begin{equation}\label{eq:directional-branch}
    \bigl(P(\tau,z)+\tau\overline{h(z)}Q(\tau,z)\bigr)^2
    =\bigl(1+\tau^2|h(z)|^2\bigr)P(\tau,z)^2 .
\end{equation}
For \(h(z)\ne0\), suppose \(\tau_0=i/|h(z)|\) lies in the disk.
Since \(P(0,z)=1\), the right side of \eqref{eq:directional-branch}
has a zero of odd order \(1+2\operatorname{ord}_{\tau_0}P(\cdot,z)\).
The left side is a holomorphic square, nonzero at the origin,
so its order at \(\tau_0\) is even. This contradiction gives \(R\le1/|h(z)|\) almost
everywhere on \(\{h\ne0\}\). As \(B\) contains the support of \(h\),
it follows that \(R_B(h)\le1/\|h\|_\infty\). The argument uses only
the Taylor coefficients of the real curve and thus also applies if the
curve is not known to equal its Taylor series.

For the reverse bound, first suppose \(0<M=\|h\|_\infty<\infty\).
For \(|\tau|<1/M\), define
\[
    \mu_\tau=\frac{\tau h}{1+\sqrt{1+\tau^2|h|^2}},
\]
using the square root with positive real part. The power series of
\(w\mapsto(1+\sqrt{1+w^2})^{-1}\) converges uniformly on every closed
subdisk of \(|w|<1\). Substituting \(w=\tau|h|\) shows that
\(\tau\mapsto\mu_\tau\) is \(L^\infty\)-holomorphic.
Moreover \(\|\mu_\tau\|_\infty\le|\tau|M<1\). The Neumann series for
\(\omega_\tau=(I-\mu_\tau\Beurling)^{-1}\mu_\tau\) therefore converges
locally uniformly in \(L^2(\C)\), and \(\omega_\tau\) vanishes off
\(\mathcal U\). The Cauchy transform on such functions is bounded from
\(L^2\) to \(W^{1,2}(B)\) by the Beurling isometry and
\cref{lem:principal-poincare} on a disk containing \(B\).
Thus \(z+\Cauchy\omega_\tau\) is a \(W^{1,2}(B)\)-valued holomorphic
extension of \(f_t|_B\) from the real interval \(|t|<1/M\). This gives
\(R_B(h)\ge1/M\) and real analyticity in each bounded direction.
For \(h=0\) the curve is constant. For unbounded \(h\) the upper bound
already gives radius zero.

Finally, choose disjoint measurable sets \(A_j\subset\mathcal U\) of
area \(c e^{-j^2}\), \(j\ge1\), with \(c>0\) sufficiently small.
The unbounded direction \(h=\sum_{j\ge1}j\mathbf 1_{A_j}\) belongs to
\(E\), since \(\sum_{j\ge1}e^{aj-j^2}<\infty\) for every \(a>0\).
Real analyticity of \(\Sol\) at zero on \(E\) would give a positive
Taylor radius along this line, contrary to \eqref{eq:directional-radius}.
\end{proof}

\section{The exterior Schwarzian and the trivial fiber}
\label{sec:bers-applications}

Take \(\mathcal U=\D\), and write \(E=E^\Phi(\D)\) and
\(\D^*=\{z\in\C:|z|>1\}\cup\{\infty\}\).
The principal solution \(f_\nu\) is univalent and conformal on \(\D^*\).
Its exterior Schwarzian is the holomorphic quadratic differential with
coefficient
\[
    \beta(\nu)=S(f_\nu|_{\D^*}),\qquad
    Sf=\frac{f'''}{f'}-\frac32\left(\frac{f''}{f'}\right)^2.
\]
For a holomorphic quadratic differential \(q\) on \(\D^*\), let
\(Q_q(\zeta)=\zeta^{-4}q(1/\zeta)\) denote its coefficient in the disk
coordinate, including its holomorphic extension to \(\zeta=0\).
Write \(\delta(\zeta)=1-|\zeta|^2\) and
\(\ell(\zeta)=\log(e/\delta(\zeta))\).
The little logarithmic Bers space \(\Berslog(\D^*)\) consists of those
\(q\) for which
\[
    \|q\|_{\log}=\sup_{\zeta\in\D}
       \frac{\delta(\zeta)^2|Q_q(\zeta)|}{\ell(\zeta)}<\infty,
    \qquad
    \lim_{|\zeta|\to1}
       \frac{\delta(\zeta)^2|Q_q(\zeta)|}{\ell(\zeta)}=0.
\]
It is a closed subspace of the corresponding weighted supremum space.

In the compact-open topology, the derivatives of \(f_\nu\) on compact exterior annuli
depend smoothly on \(\nu\), by \cref{thm:intro-smooth} and the interior
estimates for holomorphic functions. Since \(f_\nu'\) has no zeros
on the exterior, the Schwarzian also depends smoothly on \(\nu\) in that topology. Estimates at \(\zeta=0\)
follow from Cauchy's formula on a circle about zero in the disk
coordinate. Continuity in the logarithmic norm follows from a uniform boundary estimate.
The Kraus--Nehari inequality for univalent functions \cite{Nehari},
applied to \(\zeta\mapsto 1/(f_\nu(1/\zeta)-a)\) with \(a\in f_\nu(\D)\), gives
\[
    \delta(\zeta)^2|Q_{\beta(\nu)}(\zeta)|\le 6
    \qquad(\nu\in E,\ \zeta\in\D).
\]
Thus \(\beta(E)\subset\Berslog(\D^*)\). For differences, the
weighted tail on \(r<|\zeta|<1\) is at most
\(12/\log(e/(1-r^2))\), uniformly in both parameters. Compact-open
continuity on \(|\zeta|\le r\) now gives continuity in \(\|\cdot\|_{\log}\).

\subsection{The infinitesimal Bers operator}

\begin{proposition}\label{prop:infinitesimal-bers}
The compact-open differential of \(\beta\) at the origin is
\begin{equation}\label{eq:infinitesimal-bers}
    (\mathcal A\eta)(z)
    =-\frac{3}{\pi}\int_{\D}\frac{\eta(w)}{(z-w)^4}\dA(w)
    \qquad(1<|z|<\infty).
\end{equation}
As an operator from \(E\) to \(\Berslog(\D^*)\), \(\mathcal A\) is a
bounded complex-linear surjection with bounded complex-linear right inverse
\begin{equation}\label{eq:bers-right-inverse}
    (\mathcal Rq)(w)=-(1-|w|^2)^2 Q_q(\overline w).
\end{equation}
\end{proposition}

\begin{proof}
At the identity the differential of the Schwarzian is \(V\mapsto V'''\).
Differentiating the Cauchy integral for \(D\Sol_0[\eta]\)
in \eqref{eq:origin-derivative} three times gives \eqref{eq:infinitesimal-bers}.

For \(\zeta\in\D\) set
\[
    k_\zeta(w)=\frac{\delta(\zeta)^2}{\pi|1-\zeta w|^4}.
\]
The measure \(k_\zeta\,dA\) is a probability measure, since its density is
\(1/\pi\) times the Jacobian of the disk automorphism
\(w\mapsto(w-\overline\zeta)/(1-\zeta w)\).
Moreover, \(\|k_\zeta\|_\infty\le16/(\pi\delta(\zeta)^2)\), and
\eqref{eq:infinitesimal-bers} gives
\[
    \delta(\zeta)^2|Q_{\mathcal A\eta}(\zeta)|
       \le3\int_\D |\eta| k_\zeta\dA.
\]
For \(\eta\ne0\) let \(H=|\eta|/(2\|\eta\|_\Phi)\).
By \cref{lem:luxemburg-exponential}, \(\int_\D e^H\dA\le C\).
Jensen's inequality therefore gives
\[
    \int_\D H k_\zeta\dA
        \le\log\int_\D e^H k_\zeta\dA
        \le\log\bigl(C\|k_\zeta\|_\infty\bigr)
        \le C'\ell(\zeta).
\]
This proves the norm bound. For bounded \(\eta\), the probability
estimate gives \(\delta^2|Q_{\mathcal A\eta}|\le3\|\eta\|_\infty\).
Density of \(L^\infty(\D)\) in \(E\) then gives the little boundary
condition for every \(\mathcal A\eta\).

Conversely, \(|\mathcal Rq(w)|\le\|q\|_{\log}\ell(w)\), and the little
condition gives \(|\mathcal Rq(w)|=o(\ell(w))\) as \(|w|\to1\).
Since \(\int_\D e^{c\ell}\dA=\pi e^c/(1-c)\) for \(0<c<1\),
we obtain \(\|\mathcal Rq\|_\Phi\le\|\ell\|_\Phi\|q\|_{\log}\) and
\(\int_\D e^{a|\mathcal Rq|}\dA<\infty\) for every \(a>0\).
Substituting \(u=\overline w\) in \eqref{eq:infinitesimal-bers} now gives
\[
    Q_{\mathcal A\mathcal Rq}(\zeta)=
        \frac3\pi\int_\D
        \frac{(1-|u|^2)^2Q_q(u)}{(1-\zeta\overline u)^4}\dA(u)
        =Q_q(\zeta).
\]
For the last equality expand the kernel, integrate first over the
angular variable, and use
\(3\binom{n+3}{3}\int_0^1 t^n(1-t)^2\,dt=1\).
Termwise integration is justified for fixed \(\zeta\in\D\) by
\((1-|u|^2)^2|Q_q(u)|\le\|q\|_{\log}\ell(u)\in L^1(\D)\).
\end{proof}

\subsection{The obstruction at the trivial fiber}

Set \(Z=\beta^{-1}(0)\). A conformal map with vanishing Schwarzian is
M\"obius, so principal normalization gives
\(\nu\in Z\) exactly when \(f_\nu\) is the identity on \(\D^*\).

\begin{lemma}\label{lem:boundary-sequence}
There are \(c>0\) and bounded parameters \(\nu_j\in Z\) with
\(\|\nu_j\|_\Phi\to0\) such that
\begin{equation}\label{eq:boundary-lower-bound}
     \|\mathcal A\nu_j\|_{\log}\ge c\|\nu_j\|_\Phi>0.
\end{equation}
\end{lemma}

\begin{proof}
We first construct a parameter \(\nu\in Z\) with nonzero integral.
Fix \(0<b<1\) and a nonconstant real radial function
\(u(z)=\chi(|z|)\in C_c^\infty(\{|z|<b\})\), constant near zero.
For real \(t\) with \(|t|\,\Lip(u)<1\), the map
\(f_t(z)=z+tu(z)\) is bi-Lipschitz and equals the identity outside
\(\{|z|<b\}\). Put \(a(r)=\chi'(r)/2\) and \(z=re^{i\theta}\). Uniformly
on \(\D\),
\[
    \mu_{f_t}=\frac{te^{i\theta}a(r)}{1+te^{-i\theta}a(r)}
       =te^{i\theta}a-t^2a^2+O(t^3),\qquad
    F^{-1}(\mu_{f_t})=2\mu_{f_t}+O(t^3).
\]
The coefficient of \(t\) has zero angular mean. Hence
\[
    \int_\D F^{-1}(\mu_{f_t})\dA
      =-2t^2\int_\D a(|z|)^2\dA(z)+O(t^3)\ne0
\]
for all sufficiently small \(t\ne0\). Fix one such \(t\), write \(f=f_t\),
\(\nu=F^{-1}(\mu_f)\), and denote this integral by \(I\).

For \(0<s<1\) put \(m_s(w)=(w+s)/(1+sw)\), \(d_s=1-s^2\), and form
the M\"obius conjugate \(f_s=m_s\circ f\circ m_s^{-1}\) on the sphere.
It fixes a neighborhood of infinity and is the identity off
\(m_s(\{|w|<b\})\). Its parameter \(\nu_s\) therefore lies in \(Z\).
The transformation law is
\[
    \nu_s(m_s(w))=\nu(w)\frac{m_s'(w)}{\overline{m_s'(w)}},
    \qquad m_s'(w)=\frac{d_s}{(1+sw)^2}.
\]
In particular, \(\|\nu_s\|_\infty=\|\nu\|_\infty\), and the area
\(\sigma_s\) of its support is at most \(C_b d_s^2\).
By \eqref{eq:indicator-norm} and \(\Phi(t)\le e^t-1\),
\begin{equation}\label{eq:nu-s-upper}
    \|\nu_s\|_\Phi\le
       \frac{\|\nu\|_\infty}{\log(1+1/\sigma_s)}
       \le\frac{C}{\log(e/d_s)}
       \qquad(s\uparrow1).
\end{equation}
To evaluate \(Q_{\mathcal A\nu_s}(s)\), change variables by \(m_s\) in
\eqref{eq:infinitesimal-bers}. The identities
\(1-sm_s(w)=d_s/(1+sw)\) and
\(\nu_s(m_s(w))|m_s'(w)|^2=\nu(w)m_s'(w)^2\) give
\[
    d_s^2 Q_{\mathcal A\nu_s}(s)
       =-\frac{3}{\pi}\int_\D\nu(w)\dA(w)
       =-\frac{3I}{\pi}.
\]
It follows that
\begin{equation}\label{eq:Anu-s-lower}
    \|\mathcal A\nu_s\|_{\log}\ge
       \frac{3|I|}{\pi\log(e/d_s)} .
\end{equation}
Together with \eqref{eq:nu-s-upper}, this proves
\eqref{eq:boundary-lower-bound} along any sequence \(s_j\uparrow1\)
sufficiently close to one.
\end{proof}

\begin{theorem}\label{thm:trivial-fiber-obstruction}
The map \(\beta:E\to\Berslog(\D^*)\) fails to be Fr\'echet
differentiable at zero. Its trivial fiber \(Z\) is not a \(C^1\)
embedded Banach submanifold near zero. Consequently, for the
equivalence relation \(\nu\sim\eta\) defined by
\(\beta(\nu)=\beta(\eta)\), no Banach manifold structure on
\(E/{\sim}\) makes the quotient map a real \(C^1\) split submersion
near zero.
\end{theorem}

\begin{proof}
Point evaluations on \(\Berslog(\D^*)\) are continuous and separate
points, so \cref{prop:infinitesimal-bers} identifies \(\mathcal A\) as
the only possible Fr\'echet derivative at zero. For the sequence in
\cref{lem:boundary-sequence}, the identity \(\beta(\nu_j)=\beta(0)=0\)
gives a linear approximation error of at least \(c\|\nu_j\|_\Phi\).

If \(Z\) were a \(C^1\) embedded submanifold at zero, differentiating
point evaluations of \(\beta\) along curves in \(Z\) would give
\(T_0Z\subset\ker\mathcal A\). A local \(C^1\) parametrization of
\(Z\) has derivative at zero boundedly invertible onto \(T_0Z\).
Its first-order expansion thus gives
\(\operatorname{dist}_E(x,T_0Z)=o(\|x\|_\Phi)\) as \(x\to0\) in \(Z\).
But then
\(\|\mathcal Ax\|_{\log}\le\|\mathcal A\|\operatorname{dist}_E(x,T_0Z)\),
contradicting \eqref{eq:boundary-lower-bound}. The assertion about the
quotient follows from the local form of a split submersion.
\end{proof}

\section*{Acknowledgements}
The author would like to express his sincere gratitude to Professor Hiromi Otake, who has generously continued to hold seminars with him since the author's undergraduate years, and who also discussed the present work with him in seminars. The author is also grateful to Kento Sakai for attending several seminars related to this work.

% --- arXiv版のみ ---
The author also warmly thanks Takato Sakurai, Hiroto Fujisaki, and Ayumu Tanase, his junior colleagues from graduate school, for their friendship and encouragement.

\end{document}